\documentclass[11pt]{article}
\usepackage{amsmath,amsfonts,amsthm,amssymb,mathtools}
\usepackage{mathrsfs,graphicx,color,latexsym,tikz,calc,enumerate,indentfirst}
\usetikzlibrary{shadows}
\usepackage[T1]{fontenc}
\usepackage{tikz-cd}
\usetikzlibrary{patterns,arrows,decorations.pathreplacing}
\usepackage[colorlinks=true,
linkcolor=blue,citecolor=blue,
urlcolor=blue]{hyperref}

\newtheorem{theorem}{Theorem}[section]
\newtheorem{lemma}[theorem]{Lemma}

\newtheorem{conjecture}[theorem]{Conjecture}
\newtheorem{corollary}[theorem]{Corollary}
\newtheorem{remark}[theorem]{Remark}

\renewcommand\proofname{\bf{Proof}}

\allowdisplaybreaks

\title{\bf  Eigenvalue bounds for combinatorial Laplacians and an application to random complexes}

\author{
	Xiongfeng Zhan$^{a,b}$,  \ Xueyi Huang$^{b,}$\thanks{Corresponding author.}, \ Jin-Xin Zhou$^{a}$ \setcounter{footnote}{-1}\footnote{\textit{Email address:} zhanxfmath@163.com (X. Zhan), huangxy@ecust.edu.cn (X. Huang), jxzhou@bjtu.edu.cn (J.-X. Zhou).}\\[2mm]
	\small $^a$School of Mathematics and Statistics, Beijing Jiaotong University,  \\
	\small  Beijing, 100044, P. R. China\\
	\small $^b$School of Mathematics, East China University of Science and Technology, \\
	\small  Shanghai 200237, P. R. China
}

\date{}

\begin{document}
	\maketitle
	\begin{abstract}
		
		This paper establishes new eigenvalue bounds for combinatorial Laplacians of simplicial complexes, extending previous results for flag complexes by Lew (2024) and general complexes by Shukla and Yogeshwaran (2020). Using elementary matrix-theoretic methods, we derive lower bounds for the eigenvalues of the combinatorial Laplacian in terms of the graph Laplacian spectrum and combinatorial parameters that measure the deviation from a flag complex. As a consequence, we obtain upper bounds on the dimension of cohomology groups. We also generalize an eigenvalue comparison inequality between a simplicial complex and its subcomplexes to arbitrary eigenvalues. As an application of the dimension bounds, we refine a result by Kahle (2007) on the vanishing of cohomology and connectivity in the neighborhood complex of the Erd\H{o}s--R\'{e}nyi random graph.

		\par\vspace{2mm}
		
		\noindent{\bfseries Keywords:} simplicial complex, combinatorial Laplacian, cohomology, neighborhood complex, Erd\H{o}s--R\'{e}nyi random graph
		\par\vspace{1mm}
		
		\noindent{\bfseries 2010 MSC:} 05E45
	\end{abstract}

	\section{Introduction}\label{section::1}

	Let $G=(V,E)$ be a finite graph. The \textit{Laplacian} on $G$ is the matrix $L(G)\in \mathbb{R}^{V\times V}$ defined by 
	\[
	L(G)_{u,v}=
	\begin{cases}
		\deg_G(u)&\text{if $u=v$},\\
		-1  &\text{if $\{u,v\}\in E$},\\
		0  &\text{otherwise},
	\end{cases}
	\]
	for all $u,v\in V$, where $\deg_G(u)=\left|\{v\in V\mid \{u,v\}\in E\}\right|$ is the degree of $u$ in $G$. It is known that $L(G)$ is a singular positive semi-definite matrix. 
	
	Let $V$ be a finite vertex set. A  \textit{simplicial complex} $X$ on the vertex set $V$, is a collection of  subsets of $V$ that satisfies the following two conditions: 
	\begin{enumerate}[(1)]
		\item For every $v\in V$, the singleton set $\{v\}$ is in $X$.
		
		\item If $\sigma$ is in $X$, then any subset $\sigma'$ of $\sigma$ is also in $X$.
	\end{enumerate}
	An element $\sigma$ of $X$ is called a \textit{face} of $X$, and its dimension is defined as $|\sigma|-1$. A \textit{subcomplex} $X'$ of $X$ is a simplicial complex whose faces are contained in $X$. For $\sigma\in X$, the \textit{link} of $\sigma$ is the subcomplex of $X$ defined by  
	\[\operatorname{lk}_X(\sigma)=\{\tau\in X \mid \sigma\cup\tau\in X \text{ and } \sigma\cap\tau=\emptyset\}.\]
	For $k\geq -1$, let $X(k)$ denote the set of all $k$-dimensional faces of $X$, and let $f_k(X)=|X(k)|$. For $\sigma \in X(k)$, the \textit{degree} of $\sigma$ in $X$ is defined as 
	\[\deg_X(\sigma) = \left|\{\tau \in X(k+1) \mid \sigma \subset \tau\}\right|.\] 
	The $k$-\textit{skeleton} of $X$ is the subcomplex consisting of all faces of dimension at most $k$ in $X$. In particular, we denote by $G_X$ the $1$-skeleton of $X$, viewed as a graph. Furthermore, we denote by $\widetilde{H}^k(X)$ (resp. $\widetilde{H}^k(X;\mathbb{R})$) the $k$-th (reduced) cohomology group with integer (resp. real) coefficients, and by $\widetilde{H}_k(X)$ (resp. $\widetilde{H}_k(X;\mathbb{R})$) the $k$-th (reduced) homology group of $X$ with integer (resp.  real) coefficients (see \cite{Hat02} for definitions).

	Let $<$ be a fixed linear order on the vertex set $V$. 
	For $\sigma\in X(k-1)$ and $\tau\in X(k)$ satisfying $\sigma\subset \tau$, we denote 
	\[(\tau:\sigma)=(-1)^{|\{v\in \tau\mid v<u\}|},\] 
	where $u$ is the unique vertex in $\tau\setminus \sigma$. 
	For $k\geq 0$, the \textit{$k$-th real boundary map} of $X$ is the matrix $\partial_k(X)\in \mathbb{R}^{X(k-1)\times X(k)}$ defined by 
	\[
	\partial_k(X)_{\sigma,\tau}=\begin{cases}
		(\tau:\sigma)&\text{if $\sigma\subset \tau$},\\
		0& \text{otherwise},
	\end{cases}
	\]
	for all $\sigma\in X(k-1)$ and $\tau\in X(k)$. The \textit{$k$-dimensional combinatorial Laplacian} of $X$ is the matrix $L_k(X)\in \mathbb{R}^{X(k)\times X(k)}$  defined by 
	\[
	L_k(X) =\partial_{k+1}(X)\partial_{k+1}(X)^T+ \partial_k(X)^T\partial_k(X).
	\]
	By definition, $L_k(X)$ is a positive semi-definite matrix. When $k=0$, it is routine to verify that $L_0(X)=L(G_X)+J$, where $J\in \mathbb{R}^{V\times V}$ is the all-ones matrix.

	For any matrix $M \in \mathbb{R}^{n \times n}$ with real eigenvalues and $1 \leq i \leq n$, we denote by $\lambda_i(M)$ the $i$-th smallest eigenvalue of $M$. For any $k\geq 0$, let $S_{k,i}(M)$ denote the $i$-th smallest element of the multiset
	\[
	\left\{ \sum_{j \in A} \lambda_j(M) \mid A \in \binom{[n]}{k} \right\}.
	\]
	
	As a discrete analogue of the Hodge Laplacian, the concept of combinatorial Laplacians for simplicial complexes was introduced by Eckmann \cite{Eck44} in 1944. Over the past several decades, the spectra of combinatorial Laplacians have been studied extensively for various well-known families of simplicial complexes \cite{DR02,Gar73,KRS02}. For a general framework on this topic, we refer the reader to \cite{HJ13}. 
	
	In his seminal work \cite{Eck44}, Eckmann established the classical simplicial Hodge theorem, which builds a relationship between the cohomology groups and the combinatorial Laplacians of simplicial complexes.
	
	\begin{theorem}[Simplicial Hodge theorem, \cite{Eck44}]\label{thm::Hodge}
		For a simplicial complex $X$, we have
		\[\widetilde{H}^k(X;\mathbb{R})\cong \ker L_k(X).\]
	\end{theorem}

	By Theorem \ref{thm::Hodge}, establishing lower bounds for the eigenvalues of the combinatorial Laplacians of simplicial complexes is particularly important, as such bounds usually yield estimates for the dimensions of the cohomology groups.

	A simplicial complex $X$ is called a \textit{flag complex} if it consists of all cliques of its underlying graph $G_X$. In 2005, Aharoni, Berger, and Meshulam \cite{ABM05} established a recursive inequality for the smallest eigenvalues of  the combinatorial Laplacians of a flag complex $X$ on $n$ vertices:
	\[
	k\lambda_1(L_k(X)) \geq (k+1)\lambda_1(L_{k-1}(X)) - n, \quad \text{for all } k \geq 1.
	\]
	A principal consequence of this inequality is the bound 
	\[
	\lambda_1(L_k(X)) \geq (k+1)\lambda_1(L_0(X)) - kn.
	\]
	By Theorem \ref{thm::Hodge}, this implies that if $G_X$ satisfies $\lambda_2(L(G_X)) = \lambda_1(L_0(X)) > \frac{k}{k+1}n$, then $\tilde{H}^k(X; \mathbb{R}) = 0$. 
	
	Inspired by the work of Aharoni, Berger, and Meshulam \cite{ABM05}, Lew \cite{Lew20} proved that for any simplicial complex $X$ on $n$ vertices whose missing faces have maximum dimension $d$, 
	\[
	(k - d + 1)\lambda_1(L_k(X)) \geq (k + 1)\lambda_1(L_{k-1}(X)) - dn, \quad \text{for all } k \geq d.
	\]
	Using this inequality, Lew showed that if $\lambda_1(L_{d-1}(X)) > \left(1 - \binom{k+1}{d}^{-1}\right)n$, then $\tilde{H}^j(X; \mathbb{R}) = 0$ for all integers $j$ satisfying $d-1 \leq j \leq k$. 
	
	More recently, Lew \cite{Lew24} derived the following lower bounds on the eigenvalues of the combinatorial Laplacians of a flag complex $X$ on $n$ vertices using additive compound matrices.
	
	\begin{theorem}[{\cite[Theorem 1.3]{Lew24}}]\label{thm::Lew}
		Let $X$ be a flag complex on $n$ vertices. For $k \geq 0$,
		\[
		\lambda_i(L_k(X)) \geq S_{k+1,i}(L(G_X) + J) - kn.
		\] 
	\end{theorem}
	
	As a consequence, Lew proved that for all $k \geq 0$, the dimension of $\tilde{H}^k(X; \mathbb{R})$ is at most the number of $A \in \binom{[n]}{k+1}$ such that 
	\[
	\sum_{i \in A} \lambda_i(L(G_X) + J) \leq k n.
	\]

	Before proceeding, we introduce notation to quantify the discrepancy between a simplicial complex $X$ and the flag complex of its underlying graph $G_X$. For $k \geq 0$, $\sigma \in X(k)$, and $0 \leq j \leq k+1$, define
	\begin{equation}\label{equ::1}
		\sigma[j]:= \Big\{
		u \in V \Bigm| \text{$u \notin \operatorname{lk}_X(\sigma)$, $u \in \operatorname{lk}_X(v)$ for all $v \in \sigma$, and $\bigl|\{w\in \sigma: u \in \operatorname{lk}_X(\sigma \setminus \{w\})\}\bigr|=j$}\Big\}.
	\end{equation}
	Clearly, $\{\sigma[0],\ldots,\sigma[k+1]\}$ forms a partition of the set 
	\[
	\{u \in V \mid \text{$u \notin \operatorname{lk}_X(\sigma)$, $u \in \operatorname{lk}_X(v)$ for all $v \in \sigma$}\}.
	\]
	Note that $X$ is a flag complex if and only if $\sigma[j]=\emptyset$ for all $\sigma\in X(k)$, $0\leq j\leq k+1$, and $k\geq 0$.  Moreover, for $j \geq 2$, the condition ``$u \in \operatorname{lk}_X(v)$ for all $v \in \sigma$'' in \eqref{equ::1} is redundant. If this condition is removed and $k, j \geq 1$, then $\max_{\sigma \in X(k)} |\sigma[j]|$ coincides with the parameter $D_k(X,j)$ defined by Shukla and Yogeshwaran \cite{SY20}.

	
	Following the idea of Aharoni, Berger, and Meshulam \cite{ABM05} and using the parameters $D_k(X,j)$, Shukla and Yogeshwaran \cite{SY20} proved that for any simplicial complex $X$ on $n$ vertices,
	\[
	k\lambda_1(L_k(X)) \geq (k+1)\lambda_1(L_{k-1}(X)) - n - \sum_{j=2}^{k+1} \left( k(k+1) + j \right) D_k(X,j), \quad \text{for } k \geq 1.
	\]
	As a consequence, they established the following result:
	
	\begin{theorem}[{\cite[Corollary 1.7]{SY20}}]\label{cor::SY}
		Suppose the $k$-skeleton ($k \geq 1$) of $X$ coincides with that of the flag complex of $G_X$. If 
		$\lambda_2(L(G_X)) > \frac{k}{k+1}n + (k+1) D_k(X,k+1)$,
		then $\widetilde{H}^k(X;\mathbb{R}) = 0$.
	\end{theorem}
	
	It is worth noting that the results of Lew \cite{Lew20,Lew24} and Shukla and Yogeshwaran \cite{SY20} extend previous results of Aharoni, Berger, and Meshulam \cite{ABM05}. In fact, all these results can be viewed as global variants of theorems by Garland (cf. \cite[Theorem 5.9]{Gar73}) and Ballmann--\'{S}wi\k{a}tkowski (cf. \cite[Theorem 2.5]{BS97}) on the vanishing of cohomology of simplicial complexes. For further extensions and variants of this powerful method, we refer the reader to \cite{BS97,GW16,Kah14,Lew25}.

	In this paper, we first refine the notation $D_k(X,j)$ introduced by Shukla and Yogeshwaran \cite{SY20} to the more precise $\sigma[j]$, and then employ elementary matrix methods to extend the results of Theorem \ref{thm::Lew} and Theorem \ref{cor::SY} to general simplicial complexes. Our proof combines Lew's approach \cite{Lew24} with new insights regarding matrix decomposition and our notation $\sigma[j]$.

	%
	
	\begin{theorem}\label{thm::main1}
		Let $X$ be a simplicial complex on $n$ vertices. Then, for $k\geq 0$ and $1\leq i\leq f_k(X)$,
		\[
		\lambda_i({L_k(X)})\geq S_{k+1,i}(L(G_X)+J)-kn-\max_{\sigma\in X(k)}\sum_{j=0}^{k+1}(j+1)\cdot \bigl|\sigma[j]\bigr|.
		\] 
	\end{theorem}
	
	Note that by the arguments following \eqref{equ::1}, Theorem \ref{thm::main1} implies the result of Theorem \ref{thm::Lew}. Moreover, by Theorem \ref{thm::main1} and Theorem \ref{thm::Hodge}, we immediately deduce the following upper bound on the dimension of the cohomology group $\widetilde{H}^k(X;\mathbb{R})$.

	\begin{corollary}\label{cor::1}
		Let $X$ be a simplicial complex on $n$ vertices. For any $k \geq 0$,
		\[
		\dim \left( \widetilde{H}^k(X;\mathbb{R}) \right) \leq 
		\left| \left\{ A \in \binom{[n]}{k+1} \Biggm| 
		\sum_{i \in A} \lambda_i(L(G_X) + J) \leq k n + 
		\max_{\sigma \in X(k)} \sum_{j=0}^{k+1} (j+1) \cdot |\sigma[j]| 
		\right\} \right|.
		\]
	\end{corollary}
	
	Let $k \geq 1$. If the $k$-skeleton of $X$ coincides with that of the flag complex of $G_X$, 
	it is easy to verify that $\sigma[j] = \emptyset$ for all $\sigma \in X(k)$ and $0 \leq j \leq k$. 
	Observe that $(k+1)\lambda_2(L(G_X)) \leq \sum_{i \in A} \lambda_i(L(G_X) + J)$ for all $A \in \binom{[n]}{k+1}$. By Corollary \ref{cor::1}, we conclude that if
	$\lambda_2(L(G_X)) > \frac{kn}{k+1} + \frac{k+2}{k+1} \max_{\sigma \in X(k)} |\sigma[k+1]| 
	= \frac{kn}{k+1} + \frac{k+2}{k+1} D_k(X,k+1)$ then $\widetilde{H}^k(X;\mathbb{R}) = 0$, which strengthens the result of Theorem \ref{cor::SY}.
	
	In \cite{SY20}, Shukla and Yogeshwaran also derived another eigenvalue bound using variational methods for the smallest eigenvalue.
	
	\begin{theorem}[{\cite[Theorem 1.4]{SY20}}]\label{thm::SY1}
		For every simplicial complex $X$ and every subcomplex $X' \subset X$, and for $k \geq 1$, 
		\[
		\lambda_1(L_k(X')) \geq \lambda_1(L_k(X)) - (k+2) \max_{\sigma \in X'(k)} \left( \deg_{X}(\sigma) - \deg_{X'}(\sigma) \right).
		\]
	\end{theorem}

	Our matrix method can be used to generalize Theorem \ref{thm::SY1} to arbitrary eigenvalues as well.
	
	\begin{theorem}\label{thm::main2}
		Let $X$ be a simplicial complex, and let $X'$ be a subcomplex of $X$. Then, for $k\geq 0$ and $1\leq i\leq f_k(X')$,
		\[\lambda_{i}(L_k(X'))\geq \lambda_{i}(L_k(X))-(k+2)\max_{\sigma\in X'(k)}(\deg_{X}(\sigma)-\deg_{X'}(\sigma)).\]
	\end{theorem}

	Let $p := p(n) \in (0,1)$ be a monotone function of $n$. The Erd\H{o}s--R\'{e}nyi random graph $G(n,p)$ is defined as a graph on $n$ vertices where each distinct pair of vertices is connected by an edge independently with probability $p$. We say that a graph property $\mathcal{P}$ holds \textit{almost always} (\textit{a.a.}) for $G(n,p)$ if $\mathbb{P}[G(n,p) \in \mathcal{P}] \rightarrow 1$ as $n \rightarrow \infty$.
	
	Let $G = (V,E)$ be a finite graph. The \textit{neighborhood complex} of $G$, denoted by $\mathcal{N}[G]$, is the simplicial complex whose faces are the subsets $\sigma \subseteq V$ that have a common neighbor. Neighborhood complexes were introduced by Lov\'{a}sz \cite{Lov78} in his proof of the famous Kneser conjecture.
	
	A topological space $\mathcal{X}$ is said to be $k$-\textit{connected} if every map from a sphere $\mathbb{S}^i \rightarrow \mathcal{X}$ extends to a map from the ball $\mathbb{B}^{i+1} \rightarrow \mathcal{X}$ for $i = 0,1,\ldots,k$. A simplicial complex $X$ is called $k$-connected if its geometric realization is $k$-connected. 
	
	In 2007, Kahle \cite{Kah07} studied the neighborhood complex of the Erd\H{o}s--R\'{e}nyi random graph and proved the following result.
	
	\begin{theorem}[{\cite[Theorem 2.1]{Kah07}}]\label{thm::Kahle}
		Let $k \geq 0$. If $\binom{n}{k+2}(1-p^{k+2})^{n-k-2} = \operatorname{o}(1)$, then $\mathcal{N}[G(n,p)]$ is \textit{a.a.} $k$-connected.
	\end{theorem}
	
	By Theorem \ref{thm::Kahle} and the Hurewicz theorem (cf. \cite[Theorem 4.32]{Hat02}), if $p = (\frac{(k+2)\log n + c_n}{n})^{\frac{1}{k+2}}$ with $c_n = \operatorname{o}(\log n) \rightarrow \infty$, then \textit{a.a.} $\widetilde{H}^{i}(\mathcal{N}[G(n,p)];\mathbb{R})=0$ for all $i \leq k$. Using Theorem \ref{cor::SY}, Shukla and Yogeshwaran \cite{SY20} improved this result by showing that if $p = (\frac{(k+1)\log n + c_n}{n})^{\frac{1}{k+2}}$ with $c_n = \operatorname{o}(\log n) \rightarrow \infty$, then \textit{a.a.} $\widetilde{H}^{i}(\mathcal{N}[G(n,p)];\mathbb{R})=0$ for all $i \leq k$. Note that here $p$ satisfies the condition $\binom{n}{k+2}(1 - p^{k+2})^{n-k-2} = \operatorname{o}(n)$. For more related works on random simplicial complexes, we refer the reader to \cite{GW16, Kah14, Kah14+}.
	
	In this paper, using Theorem \ref{thm::Kahle} and Corollary \ref{cor::1}, we derive the following result, which refines the works of Kahle \cite{Kah07} and Shukla and Yogeshwaran \cite{SY20}.
	
	\begin{theorem}\label{thm::main3}
		Let $k \geq s \geq 1$. If $\binom{n}{k+2}(1 - p^{k+2})^{n-k-2} = \operatorname{o}(n^s)$, then $\mathcal{N}[G(n,p)]$ is \textit{a.a.} $(k-s)$-connected, and \textit{a.a.} $\widetilde{H}^{k-s+1}(\mathcal{N}[G(n,p)])=0$.
	\end{theorem}

	The paper is organized as follows. In Section \ref{section::2}, we express $L_k(X)$ as the difference of two matrices and provide a concise proof of Theorem \ref{thm::main1} using elementary matrix techniques. Theorem \ref{thm::main2} is also established via similar matrix methods. Section \ref{section::3} is devoted to the proof of Theorem \ref{thm::main3}, which relies on Theorem \ref{thm::Kahle} and Corollary \ref{cor::1}. Finally, in Section \ref{section::4}, we propose two conjectures related to Theorem \ref{thm::main3} for further research.
	
	\section{Eigenvalue bounds for combinatorial Laplacians}\label{section::2}
	%
	In this section, we prove Theorems \ref{thm::main1} and \ref{thm::main2} using elementary matrix methods.
	
	Let $X$ be a simplicial complex on vertex set $V$ with $|V|=n$ and $k \geq 0$. For any $\sigma, \tau \in X(k)$, we write $\sigma \sim \tau$ if  $|\sigma \cap \tau| = k$ and $\sigma \cup \tau \notin X(k+1)$. Moreover, when $\sigma \sim \tau$, we define:
	\begin{itemize}
		\item $\sigma \sim_1 \tau$ if $(\sigma \setminus \tau) \cup (\tau \setminus \sigma) \in X$;
		\item $\sigma \sim_2 \tau$ if $(\sigma \setminus \tau) \cup (\tau \setminus \sigma) \notin X$.
	\end{itemize}
	Then the matrix $L_k(X)$ has the following explicit form:

	\begin{lemma}[{\cite{DR02}}]\label{lem::Matrix_Rep}
		Let $k\geq 0$. Then
		\begin{equation*}
			L_k(X)(\sigma,\tau)=\begin{cases} 
				\deg_X(\sigma)+k+1 & \mbox{if $\sigma=\tau$}, \\
				(\sigma:\sigma\cap \tau)(\tau:\sigma\cap \tau) &\mbox{if $\sigma\sim\tau$}, \\
				0 & \mbox{otherwise},
			\end{cases}   
		\end{equation*}
		for all $\sigma,\tau\in X(k)$.
	\end{lemma}
	We now decompose the matrix $L_k(X)$ as  
	\[L_k(X)=Q-P,\] 
	where the matrices
	$P$ and $Q$ are defined as:
	\begin{align}
		P(\sigma,\tau)&=\begin{cases} 
			\sum_{v\in \sigma}\deg_{G_X}(v)-\deg_X(\sigma) & \mbox{if $\sigma=\tau$}, \\
			-(\sigma:\sigma\cap \tau)(\tau:\sigma\cap \tau) &\mbox{if $\sigma\sim_1\tau$}, \\
			0 & \mbox{otherwise},
		\end{cases}\label{equ::P}\\
		Q(\sigma,\tau)&=\begin{cases} 
			\sum_{v\in \sigma}\deg_{G_X}(v)+k+1 & \mbox{if $\sigma=\tau$}, \\
			(\sigma:\sigma\cap \tau)(\tau:\sigma\cap \tau) &\mbox{if $\sigma\sim_2\tau$}, \\
			0 & \mbox{otherwise},
		\end{cases}\label{equ::Q} 
	\end{align}
	for all $\sigma,\tau\in X(k)$.
	
	To prove Theorem \ref{thm::main1}, we need the following upper bound on the largest eigenvalue of the matrix $P$ defined in \eqref{equ::P}.
	
	\begin{lemma}\label{lem::key1}
		For $k\geq 0$, 
		\[\lambda_{f_k(X)}(P)\leq kn+\max_{\sigma\in X(k)}\sum_{j=0}^{k+1}(j+1)\bigl|\sigma[j]\bigr|.\]
	\end{lemma}
	\begin{proof}
		For any $\sigma \in X(k)$, we shall prove that
		\begin{equation}\label{equ::2}
			\sum_{v \in \sigma} \deg_{G_X}(v) \leq kn + \deg_X(\sigma) + \sum_{j=0}^{k+1} |\sigma[j]|.
		\end{equation}
		Observe that the left-hand side of \eqref{equ::2} counts the number of ordered pairs $(u,v)$ such that $u \in V$, $u \in \operatorname{lk}(v)$, and $v \in \sigma$. If $u \in \operatorname{lk}(\sigma)$, then $u \in \operatorname{lk}(v)$ for all $v \in \sigma$, and so $u$ appears in exactly $k+1$ such pairs. For $u \notin \operatorname{lk}(\sigma)$, we have two possibilities: either $u \in \bigcup_{j=0}^{k+1} \sigma[j]$ or $u \notin \bigcup_{j=0}^{k+1} \sigma[j]$. In the former case, by the definition of $\sigma[j]$, we see that $u$ appears in exactly $k+1$ such pairs. In the latter case, we claim that $u$ appears in at most $k$ such pairs, since otherwise we would obtain $u \in \bigcup_{j=0}^{k+1} \sigma[j]$, a contradiction. Therefore, 
		\[
		\begin{aligned}
			\sum_{v\in \sigma}\deg_{G_X}(v)&\leq (k+1)\left(\deg_X(\sigma)+\sum_{j=0}^{k+1}|\sigma[j]|\right)+k\left(n-\deg_X(\sigma)-\sum_{j=0}^{k+1}|\sigma[j]|\right)\\
			&=kn+\deg_X(\sigma)+\sum_{j=0}^{k+1}|\sigma[j]|,
		\end{aligned}
		\]
		and \eqref{equ::2} follows.	
		
		Now fix $\sigma \in X(k)$. Let $\tau \in X(k)$ be such that $\sigma \sim_1 \tau$, and let $u$ be the unique vertex in $\tau \setminus \sigma$. Then $u \notin \operatorname{lk}(\sigma)$, and $u \in \operatorname{lk}(v)$ for all $v \in \sigma$. We thus assert that $u \in \sigma[j]$ for some $1\leq j\leq k+1$. On the other hand, for $1 \leq j \leq k+1$ and $u \in \sigma[j]$, there are exactly $j$ faces in $X(k)$, say $\tau_1^u, \ldots, \tau_j^u$, such that $\sigma \sim_1 \tau_i^u$ and $\tau_i^u \setminus \sigma=\{u\}$ for $1 \leq i \leq j$. Note that for any two distinct $u \in \sigma[j]$ and $u' \in \sigma[j']$, the two sets $\{\tau_1^u, \ldots, \tau_j^u\}$ and $\{\tau_1^{u'}, \ldots, \tau_{j'}^{u'}\}$ are disjoint. Therefore, we conclude that the number of $\tau \in X(k)$ such that $\sigma \sim_1 \tau$ is
		\begin{equation}\label{equ::3}
			\bigl|\{\tau \in X(k) : \sigma \sim_1 \tau\}\bigr| = \sum_{j=1}^{k+1} j |\sigma[j]| = \sum_{j=0}^{k+1} j |\sigma[j]|.
		\end{equation}

		From \eqref{equ::2}, \eqref{equ::3} and the Ger\v{s}gorin circle theorem (cf. \cite[Theorem 6.1.1]{Hor13}), we obtain
		\begin{equation*}
			\begin{aligned}
				\lambda_{f_k(X)}(P) &\leq \max_{\sigma \in X(k)} \left( \sum_{v \in \sigma} \deg_X(v) - \deg_X(\sigma) + |\{\tau \in X(k) : \sigma \sim_1 \tau\}| \right) \\
				&\leq \max_{\sigma \in X(k)} \left( kn + \sum_{j=0}^{k+1} (j+1) |\sigma[j]| \right) \\
				&= kn + \max_{\sigma \in X(k)} \sum_{j=0}^{k+1} (j+1) \bigl| \sigma[j] \bigr|,
			\end{aligned}
		\end{equation*}
		as desired.
	\end{proof}

	We also need the notation of additive compound matrices. Let $M$ be an $n \times n$ matrix over a field $\mathbb{F}$, and let $1 \leq k \leq n$. The \textit{$k$-th additive compound} of $M$ is the matrix  $M^{[k]}\in \mathbb{F}^{\binom{[n]}{k} \times \binom{[n]}{k}}$ defined by
	\begin{equation}\label{equ::4}
		M^{[k]}(\sigma, \tau) = 
		\begin{cases} 
			\sum_{i \in \sigma} M(i, i) & \text{if $\sigma = \tau$}, \\
			(\sigma:\sigma\cap \tau)(\tau:\sigma\cap \tau) \cdot M(i, j) & \text{if $|\sigma \cap \tau| = k - 1$, $\sigma \setminus \tau = \{i\}$, and $\tau \setminus \sigma = \{j\}$}, \\
			0 & \text{otherwise},
		\end{cases}
	\end{equation}
	for all $\sigma,\tau\in \binom{[n]}{k}$.
	\begin{lemma}[{\cite[Theorem 2.1]{Fie74}}]\label{lem::add_matrix}
		Let $M$ be an $n\times n$ matrix over a field $\mathbb{F}$, with eigenvalues $\lambda_1,\ldots,\lambda_n$. Then, the $k$-th additive compound $M^{[k]}$ has eigenvalues $\lambda_{i_1}+\cdots+\lambda_{i_k}$, for $1\leq i_1<\cdots<i_k\leq n$.
	\end{lemma} 	
	
	Now we are in a position to give the proof of Theorem \ref{thm::main1}. 
	
	\renewcommand\proofname{\bf{Proof of Theorem \ref{thm::main1}}}
	\begin{proof}
		
		Let $V$ be the vertex set of $X$. Recall that the underlying graph $G_X$ has vertex set $V$ and edge set $X(1)$. We have 
		\begin{equation}\label{equ::5}
			(L(G_X) + J)(u,v) = \begin{cases} 
				\deg_{G_X}(u) + 1 & \text{if } u = v, \\
				1 & \text{if } \{u,v\} \notin X(1), \\
				0 & \text{otherwise},
			\end{cases}
		\end{equation}
		for all $u,v \in V$.
		
		For $\sigma, \tau \in X(k)$ with $|\sigma \cap \tau| = k$, denote $(\sigma \setminus \tau) \cup (\tau \setminus \sigma) = \{u, v\}$. If $\{u, v\} \notin X(1)$, then $\sigma \cup \tau \notin X(k+1)$, and hence $\sigma \sim_2 \tau$. Conversely, if $\sigma \sim_2 \tau$, then $\{u, v\} \notin X(1)$. Therefore, $\{u, v\} \notin X(1)$ if and only if $\sigma \sim_2 \tau$.
		
		Combining \eqref{equ::4} and \eqref{equ::5}, we conclude that the matrix $Q$ defined in \eqref{equ::Q} is a principal submatrix of $(L(G_X) + J)^{[k+1]}$, with rows and columns indexed by $X(k)$. Therefore, by the Cauchy interlacing theorem (cf. \cite[Theorem 4.3.17]{Hor13}) and Lemma \ref{lem::add_matrix}, we obtain
		\begin{equation}\label{equ::6}
			\lambda_i(Q) \geq \lambda_i((L(G_X) + J)^{[k+1]}) = S_{k+1,i}(L(G_X) + J).
		\end{equation}
		Then it follows from the Weyl inequality (cf. \cite[Theorem 4.3.1]{Hor13}), Lemma \ref{lem::key1} and \eqref{equ::6} that
		\begin{equation*}
			\begin{aligned}
				\lambda_i(L_k(X)) &\geq \lambda_i(Q) - \lambda_{f_k(X)}(P) \\
				&\geq S_{k+1,i}(L(G_X) + J) - kn - \max_{\sigma \in X(k)} \sum_{j=0}^{k+1} (j+1) \bigl| \sigma[j] \bigr|.
			\end{aligned}
		\end{equation*}
		
		This completes the proof.
	\end{proof}
	
	Next, we present the proof of Theorem \ref{thm::main2}.
	
	\renewcommand\proofname{\bf{Proof of Theorem \ref{thm::main2}}}
	\begin{proof}
		Let $L'$ be the principal submatrix of $L_k(X)$ with rows and columns indexed by $X'(k)$. 
		Then, by Lemma \ref{lem::Matrix_Rep}, we have
		\begin{equation*}
			(L' - L_k(X'))(\sigma,\tau) = \begin{cases} 
				\deg_X(\sigma) - \deg_{X'}(\sigma) & \text{if } \sigma = \tau, \\
				-(\sigma:\sigma\cap \tau)(\tau:\sigma\cap \tau) & \text{if } |\sigma \cap \tau| = k,\ \sigma \cup \tau \in X(k+1) \setminus X'(k+1), \\
				0 & \text{otherwise},
			\end{cases}
		\end{equation*}
		for all $\sigma, \tau \in X'(k)$. Note that
		\[
		\left|\left\{ \tau : |\sigma \cap \tau| = k,\ \sigma \cup \tau \in X(k+1) \setminus X'(k+1) \right\}\right| = (k+1)\left(\deg_X(\sigma) - \deg_{X'}(\sigma)\right).
		\]
		By the Ger\v{s}gorin circle theorem, we obtain
		\[
		\lambda_{f_k(X')}(L' - L_k(X')) \leq (k+2) \max_{\sigma \in X'(k)} \left(\deg_X(\sigma) - \deg_{X'}(\sigma)\right).
		\]
		Then, applying the Weyl inequality and the Cauchy interlacing theorem yields
		\begin{equation*}
			\begin{aligned}
				\lambda_i(L_k(X')) &\geq \lambda_i(L') - \lambda_{f_k(X')}(L' - L_k(X')) \\
				&\geq \lambda_i(L_k(X)) - (k+2) \max_{\sigma \in X'(k)} \left(\deg_X(\sigma) - \deg_{X'}(\sigma)\right).
			\end{aligned}
		\end{equation*}	
		
		This completes the proof.
	\end{proof}
	
	\begin{remark}\rm
Let $X$ be a simplicial complex on $n$ vertices, and let $Y$ be the flag complex of $G_X$. Note that $X$ is a subcomplex of $Y$. Combining Theorem \ref{thm::main2} and Theorem \ref{thm::Lew}, we obtain
\[
\lambda_i(L_k(X)) \geq S_{k+1,i}(L(G_X) + J) - kn - (k+2) \max_{\sigma \in X(k)} \left( \deg_{Y}(\sigma) - \deg_{X}(\sigma) \right).
\]
However, this bound is weaker than the one in Theorem \ref{thm::main1} since $\deg_{Y}(\sigma) - \deg_{X}(\sigma) = \sum_{j=0}^{k+1} |\sigma[j]|$.
	\end{remark}

	\section{An application to random complexes}\label{section::3}
	In this section, we use Theorem \ref{thm::Kahle} and Corollary \ref{cor::1} to prove Theorem \ref{thm::main3}.
	
	\begin{lemma}\label{lem::Order}
		Let $p=p(n)\in (0,1)$. For any fixed integers $a,b$ with $1\leq a\leq b$, we have
		\[
		\frac{\binom{n}{a}(1-p^a)^{n-a}}{\binom{n}{b}(1-p^b)^{n-b}}=\mathcal{O}(n^{-(b-a)}).
		\]
	\end{lemma}
	\renewcommand\proofname{\bf{Proof}}
	\begin{proof}
		Note that if $a=b$ then the result is trivial. Thus we may assume that $a<b$. Let $G(n,p)$ be the Erd\H{o}s--R\'{e}nyi random graph on vertex set $V$, and let $X = \mathcal{N}[G(n,p)]$. Denote $\overline{X(k)} = \binom{V}{k+1} \setminus X(k)$. Since the expectation of the number of $(k+1)$-tuples of vertices in $G(n,p)$ that have no common neighbor is $\binom{n}{k+1}(1 - p^{k+1})^{n-k-1}$, we have
		\begin{equation}\label{equ::7}
			\mathbb{E}\left[|\overline{X(k)}|\right] = \binom{n}{k+1}(1 - p^{k+1})^{n-k-1}.
		\end{equation}
		Let $N$ denote the number of ordered pairs $(\sigma, \tau)$ such that $\sigma \in \overline{X(k)}$, $\tau \in \overline{X(k+1)}$, and $\sigma \subset \tau$. For every $\sigma \in \overline{X(k)}$, there are exactly $n - k - 1$ sets $\tau \in \overline{X(k+1)}$ such that $\sigma \subset \tau$. Therefore, we have
		\[
		N = (n - k - 1)|\overline{X(k)}|.
		\]
		On the other hand, for every $\tau \in \overline{X(k+1)}$, there are at most $k + 2$ sets $\sigma \in \overline{X(k)}$ such that $\sigma \subset \tau$. Hence,
		\[
		N \leq (k + 2)|\overline{X(k+1)}|,
		\]
		and it follows that
		\[
		\frac{|\overline{X(k)}|}{|\overline{X(k+1)}|} \leq \frac{k + 2}{n - k - 1}.
		\]
		Taking expectations and noting the linearity of expectation, we find that for any fixed $k$,
		\[
		\frac{\mathbb{E}\left[|\overline{X(k)}|\right]}{\mathbb{E}\left[|\overline{X(k+1)}|\right]} = \mathcal{O}(n^{-1}).
		\]
		By iterating this argument, we obtain for $a < b$,
		\[
		\frac{\mathbb{E}\left[|\overline{X(a)}|\right]}{\mathbb{E}\left[|\overline{X(b)}|\right]} = \mathcal{O}(n^{-(b - a)}).
		\]
		Combining this asymptotic bound with the explicit formula from \eqref{equ::7}, the result follows. 
	\end{proof}

	\renewcommand\proofname{\bf{Proof of Theorem \ref{thm::main3}}}
	\begin{proof}
		
		According to Lemma \ref{lem::Order}, it suffices to consider the case $s = 1$. Suppose $k \geq 1$ and $\binom{n}{k+2}(1 - p^{k+2})^{n-k-2} = o(n)$. By Lemma \ref{lem::Order} again, we have $\binom{n}{k+1}(1 - p^{k+1})^{n-k-1} = o(1)$. Thus, by Theorem \ref{thm::Kahle}, $\mathcal{N}[G(n,p)]$ is \textit{a.a.} $(k-1)$-connected. Then it follows from the Hurewicz theorem (cf. \cite[Theorem 4.32]{Hat02}) that
		\begin{equation}\label{equ::8}
			\textit{a.a.}\quad \widetilde{H}_{k-1}(\mathcal{N}[G(n,p)]) = 0. 
		\end{equation}
		
		Note that for any $\sigma \in X(k)$ and $u \in \bigcup_{i=0}^{k+1} \sigma[i]$, we have $\sigma \cup \{u\} \notin X(k+1)$. Hence,
		\[
		\max_{\sigma \in X(k)} \sum_{j=0}^{k+1} \bigl| \sigma[j] \bigr| \leq |\overline{X(k+1)}|.
		\]
		Combining this with \eqref{equ::7} yields
		\[
		\mathbb{E} \left[ \max_{\sigma \in X(k)} \sum_{j=0}^{k+1} \bigl| \sigma[j] \bigr| \right] \leq \binom{n}{k+2}(1 - p^{k+2})^{n-k-2} = o(n).
		\]
		Let $\Delta(k) = \max_{\sigma \in X(k)} \sum_{j=0}^{k+1} (j+1) \bigl| \sigma[j] \bigr|$. Then $\mathbb{E}[\Delta(k)] = o(n)$. By Markov's inequality, $\mathbb{P}[\Delta(k) \geq n] = o(1)$, \textit{i.e.},
		\begin{equation}\label{equ::9}
			\textit{a.a.}\quad \Delta(k) < n.
		\end{equation}
		
		Recall that $k\geq 1$ and $\binom{n}{k+2}(1 - p^{k+2})^{n-k-2} = o(n)$. By \eqref{equ::7} and Lemma \ref{lem::Order}, we have
		\[
		\mathbb{E}\left[|\overline{X(1)}|\right] = \binom{n}{2}(1 - p^2)^{n-2} = o(1).
		\]
		Again by Markov’s inequality, $\mathbb{P}[|\overline{X(1)}| \geq 1] = o(1)$. Hence, $G_X$ is \textit{a.a.} a complete graph, and so
		\[
		\textit{a.a.}\quad \sum_{i \in A} \lambda_i(L(G_X) + J) = (k+1)n
		\]
		for any $A \in \binom{[n]}{k+1}$. Then from Corollary~\ref{cor::1} and \eqref{equ::9}, we obtain that $\widetilde{H}^{k}(\mathcal{N}[G(n,p)];\mathbb{R})$ \textit{a.a.} vanishes. Combining this with \eqref{equ::8} and applying the universal coefficient theorem for cohomology (cf.~\cite[Theorem 3.2]{Hat02}), we conclude that 
		\[
		\textit{a.a.}\quad  \widetilde{H}^{k}(\mathcal{N}[G(n,p)])=0.\] 
		
		This completes the proof.
	\end{proof}
	
	\section{Further research}\label{section::4}
	
	Let $X$ be a simplicial complex on $n$ vertices and $k \geq 1$. The Hurewicz theorem (cf. \cite[Theorem 4.32]{Hat02})  indicates that $X$ is $k$-connected if and only if it is $1$-connected and $\widetilde{H}_i(X) = 0$ for all $i \leq k$. It is therefore natural to expect that the result of Theorem \ref{thm::main3} can be strengthened as follows.
	\begin{conjecture}
		Let $k \geq s \geq 1$. If $\binom{n}{k+2}(1-p^{k+2})^{n-k-2} = o(n^s)$, then $\mathcal{N}[G(n,p)]$ is \textit{a.a.} $(k-s+1)$-connected.
	\end{conjecture}
	
	We also conjecture that the result of Theorem \ref{thm::main3} can be extended to the case $s=0$.
	\begin{conjecture}
		Let $k \geq 0$. If $\binom{n}{k+2}(1-p^{k+2})^{n-k-2} = o(1)$, then \textit{a.a.} $\widetilde{H}^{k+1}(\mathcal{N}[G(n,p)]) = 0$.
	\end{conjecture}

	\section*{Declaration of Interest Statement}
	
	The authors declare that they have no known competing financial interests or personal relationships that could have appeared to influence the work reported in this paper.

	\section*{Acknowledgements}
	
	The authors are grateful to Professor Alan Lew for his valuable comments and helpful suggestions. X. Huang was supported by the National Natural Science Foundation of China (Grant No. 12471324) and the Natural Science Foundation of Shanghai (Grant No. 24ZR1415500). J.-X. Zhou was supported by the National Natural Science Foundation of China (Grants Nos. 12425111, 12071023, 12331013, 12161141005).
	
	\section*{Data Availability}
	No data was used for the research described in the article.


\begin{thebibliography}{99}\setlength{\itemsep}{0pt}
		\bibitem{ABM05} R. Aharoni, E. Berger, R. Meshulam, Eigenvalues and homology of flag complexes and vector representations of graphs, Geom. Funct. Anal. 15(3) (2005) 555--566.
		
		\bibitem{BS97} W. Ballmann, J. \'{S}wi\k{a}tkowski, On $L^2$-cohomology and property $(T)$ for automorphism groups of polyhedral cell complexes, Geom. Funct. Anal. 7(4) (1997) 615--645.
		
		\bibitem{DR02} A. Duval, V. Reiner, Shifted simplicial complexes are Laplacian integral, Trans. Amer. Math. Soc. 354(11) (2002) 4313--4344.
		
		\bibitem{Eck44} B. Eckmann, Harmonische funktionen und randwertaufgaben in einem komplex, Comment. Math. Helv. 17(1) (1944) 240--255.
		
		\bibitem{Fie74} M. Fiedler, Additive compound matrices and an inequality for eigenvalues of symmetric stochastic matrices, Czechoslovak Math. J. 24(99) (1974) 392--402.
		
		
		\bibitem{Gar73} H. Garland, $p$-adic curvature and the cohomology of discrete subgroups of $p$-adic groups, Ann. of Math. (2) 97 (1973) 375--423.
		
		\bibitem{GW16} A. Gundert, U. Wagner, On eigenvalues of random complexes, Israel J. Math. 216(2) (2016) 545--582.
		
		\bibitem{Hat02} A. Hatcher, Algebraic Topology, Cambridge University Press, Cambridge, 2002.
		
		\bibitem{HJ13} D. Horak, J. Jost, Spectra of combinatorial Laplace operators on simplicial complexes, Adv. Math. 244 (2013) 303--336.
		
		\bibitem{Hor13} R. A. Horn, C. R. Johnson, Matrix Analysis (2ed Ed.), Cambridge University Press,  Cambridge, 2013.
		
		\bibitem{Kah07} M. Kahle, The neighborhood complex of a random graph, J. Combin. Theory Ser. A 114(2) (2007) 380--387.
		
		\bibitem{Kah14} M. Kahle, Sharp vanishing thresholds for cohomology of random flag complexes, Ann. of Math. (2) 179(3) (2014) 1085--1107.
		
		\bibitem{Kah14+}  M. Kahle, Topology of random simplicial complexes: a survey, in: Algebraic Topology: Applications and New Directions, in: Contemp. Math., vol. 620, Amer. Math. Soc., Providence, RI, 2014,
		pp. 201--221.
		
		\bibitem{KRS02} W. Kook, V. Reiner, D. Stanton, Combinatorial Laplacians of matroid complexes, J. Amer. Math. Soc. 13 (2002) 129--148.
		
		\bibitem{Lew20} A. Lew, The spectral gaps of generalized flag complexes and a geometric Hall-type theorem, Int. Math. Res. Not. 2020(11) (2020) 3364--3395.
		
		
		\bibitem{Lew24} A. Lew, Laplacian eigenvalues of independence complexes via additive compound matrices, Discrete Anal. (2024) 15.
		
		\bibitem{Lew25} A. Lew, An eigenvalue interlacing approach to Garland's method, 2025, \url{https://arxiv.org/abs/2508.17279v1}.
		
		\bibitem{Lov78} L. Lov\'{a}sz, Kneser's conjecture, chromatic number and homotopy, J. Combin. Theory Ser. A 25(3) (1978) 319--324.
		
		\bibitem{SY20} S. Shukla, D. Yogeshwaran, Spectral gap bounds for the simplicial Laplacian and an application to random complexes, J. Combin. Theory Ser. A 169 (2020) 105134.
		
		
	\end{thebibliography}
\end{document}